\documentclass{amsart}
\usepackage{amsmath, amssymb}
\usepackage{mathtools}
\usepackage{hyperref}
\usepackage[portrait, margin=1in]{geometry}
\usepackage[foot]{amsaddr}
\usepackage{xcolor}
\hypersetup{
    colorlinks=true,
    linktoc=all,
    linkcolor=blue,
    citecolor=blue
}

\usepackage{todonotes}

\usepackage{tikz}

\newtheorem{theorem}{Theorem}

\newtheorem{lemma}[theorem]{Lemma}

\newtheorem*{theorem*}{Theorem}
\newtheorem*{corollary*}{Corollary}
\newtheorem*{lemma*}{Lemma}
\newtheorem*{prop*}{Proposition}
\newtheorem*{example*}{Example}

\theoremstyle{definition}

\newtheorem*{defn*}{Definition}
\newtheorem*{remark*}{Remark}

\newcommand{\floor}[1]{\left\lfloor #1 \right\rfloor}

\title{Asymptotic Properties of Maximal $p$-Core $p'$-Partitions}
\author{Sanjana Das}
\date{August 20, 2022}
\email{\href{mailto:sanjanad@mit.edu}{sanjanad@mit.edu}}
\subjclass[2020]{11P99, 05C12, 20C30}
\keywords{$p$-core partitions, $p'$-partitions, abaci, walks on graphs, character tables}

\begin{document}

\begin{abstract}
    For primes $p$, we study the maximal possible size of a $p$-core $p'$-partition (a partition with no hook lengths or parts divisible by $p$). McDowell recently proved that the maximum is attained by a unique partition, say $\Lambda_p$. Using his graph-theoretic description of $\Lambda_p$, we prove for $p > 10^6$ that \[\frac{1}{24}p^6 - p^5\sqrt{p} < |\Lambda_p| < \frac{1}{24}p^6 - \frac{1}{200}p^5\sqrt{p},\] which shows that $|\Lambda_p| \sim p^6/24$ as $p \to \infty$. 
\end{abstract}

\maketitle

\section{Introduction}

A partition is called \emph{$p$-core} if none of its hook lengths are divisible by $p$; a partition is called a \emph{$p'$-partition} (or \emph{$p$-regular}) if none of its parts are divisible by $p$. In this paper, we study the largest possible size of a $p$-core $p'$-partition. 

Interest in $p$-core $p'$-partitions arises from studying zeros in the character table of the symmetric group $S_n$. In particular, McSpirit and Ono \cite{MO22} studied entries of the character table indexed by two $p$-core partitions $\lambda$ and $\mu$. Using the result that $\chi_\lambda(\mu) = 0$ whenever $\mu$ has a part which is not a hook length in $\lambda$, they answered a question of McKay by studying pairs $(\lambda, \mu)$ of $p$-core partitions, where $\mu$ is also a $p'$-partition (see \cite[Theorem 1.3]{MO22}). For each prime $p$, they proved that there are finitely many $p$-core $p'$-partitions --- in particular, they showed that any such partition $\mu$ must satisfy \[|\mu| \leq \frac{1}{24}(p^6 - 2p^5 + 2p^4 - 3p^2 + 2p).\] Using this result, they concluded that for all sufficiently large $n$, any entry in the character table of $S_n$ indexed by two $p$-core partitions must be $0$, providing a lower bound for the number of zeros indexed by $p$-core partitions. 

The problem of determining the largest possible size of a $p$-core $p'$-partition was recently further studied by McDowell \cite{McD22}, who proved that any $p$-core $p'$-partition $\lambda$ must satisfy 
\begin{align}
    |\lambda| \leq \frac{1}{24}(p^6 - 4p^5 + 5p^4 + 12p^3 - 42p^2 + 52p - 24), \label{McD upper bound}
\end{align} 
and that there exists a $p$-core $p'$-partition $\lambda$ of size \[|\lambda| = \frac{1}{96}(p^6 + 6p^4 - 12p^3 + 89p^2 - 120p - 48).\] Furthermore, McDowell obtained an elegant graph-theoretic criterion describing the maximal $p$-core $p'$-partition (see \cite[Theorem 4.8]{McD22}), and showed that such a partition is unique for every odd prime $p$. We denote this unique partition by $\Lambda_p$. 

In view of these results, it is natural to study the asymptotics as $p \to \infty$ of $|\Lambda_p|$. Our work makes use of McDowell's description to improve the upper and lower bounds above. Our main result explicitly restricts $|\Lambda_p|$ to a narrow interval around $\frac{1}{24} p^6$.  

\begin{theorem} \label{main theorem}
    For all primes $p > 10^6$, we have \[\frac{1}{24}p^6 - p^5\sqrt{p} < |\Lambda_p| < \frac{1}{24}p^6 - \frac{1}{200}p^5\sqrt{p}.\] 
\end{theorem}

\begin{remark*}
    The range $p > 10^6$ was chosen so that the constants in the theorem above are nice. 
\end{remark*}

This paper is organized as follows. First, we describe the characterization given by McDowell, which represents $\Lambda_p$ in terms of the longest walk on the additive residue graph $\mathcal{G}_p$. The challenge is to recover the size of $\Lambda_p$ from this description. After recalling this criterion, we establish a few lemmas about the sizes of individual pieces of this longest walk. Finally, we combine these lemmas to establish the upper and lower bounds on $|\Lambda_p|$. 

\subsection*{Acknowledgements}

The author would like to thank Ken Ono and Eleanor McSpirit for their guidance, and Eoghan McDowell for helpful comments on an earlier version of this article. The author also thanks the anonymous referees and Kathrin Bringmann for further suggestions that improved the presentation of this article. The author was a participant in the 2022 UVA REU in Number Theory, and is grateful for the support of grants from the National Science Foundation
(DMS-2002265, DMS-2055118, DMS-2147273), the National Security Agency (H98230-22-1-0020), and the Templeton World Charity Foundation.

\section{A Graph-Theoretic Interpretation and Proof of Theorem \ref{main theorem}}

\subsection{Abacus Notation and a Graph-Theoretic Interpretation}

To describe $p$-core $p'$-partitions, we use \emph{abacus notation} on a $p$-abacus, as described in \cite[Chapter 2]{JK84}. The abacus has $p$ vertical runners, labelled $0$ through $p - 1$; positions on the abacus are ordered from left to right and top to bottom (starting with position $0$), so that the $i$th row on runner $j$ is position $(i - 1)p + j$. 

On this abacus, we place some \emph{beads}, and call positions without beads \emph{gaps}; we require position $0$ to be a gap. Any such abacus corresponds to the partition where each bead contributes a part equal to the number of gaps preceding the bead. Under this correspondence, every partition can be represented by a unique configuration on the abacus. 

\begin{lemma}[{\cite[Lemma 2.7.13]{JK84}}] \label{top alignment}
    A partition is $p$-core if and only if all beads of its corresponding abacus are at the top of their runners (in other words, there is no bead with a gap above it). 
\end{lemma}

Given any $p$-core partition, for each $1 \leq i \leq p - 1$, we define the $i$th \emph{bead multiplicity} $b_i$ to be the number of beads on runner $i$. By Lemma \ref{top alignment}, the list of bead multiplicities uniquely determines a $p$-core partition. 

Given the bead multiplicities of a partition, it is possible to explicitly compute the size of the partition. 

\begin{lemma}[{\cite[Lemma 2.2]{McD22}}] \label{size from bi}
    If $\lambda$ is a $p$-core partition with bead multiplicities $(b_1, b_2, \ldots, b_{p - 1})$, then \[|\lambda| = -\frac{1}{2}\left(\sum_{i = 1}^{p - 1} b_i\right)^2 + \frac{p}{2}\sum_{i = 1}^{p - 1} b_i^2 + \sum_{i = 1}^{p - 1} \left(i - \frac{p - 1}{2}\right)b_i.\] 
\end{lemma}

\begin{remark*}
    The lemma above is a minor reformulation of \cite[Lemma 2.2]{McD22}. 
\end{remark*}

It is possible to further describe the abacus configuration corresponding to a maximal $p$-core $p'$-partition. 

\begin{lemma}[{\cite[Lemma 2.3]{McD22}}] \label{right alignment}
    The abacus corresponding to any maximal $p$-core $p'$-partition has all beads at the right end of their rows (in other words, there is no bead with a gap to its right). 
\end{lemma}

We call an abacus \emph{aligned} if it has the properties given in Lemmas \ref{top alignment} and \ref{right alignment}, i.e., if all beads are topmost in their runners and rightmost in their rows. Given any aligned abacus, in a row with $i$ gaps followed by $p - i$ beads, all beads will contribute a part of the same size; this size is exactly $i$ greater than the parts contributed by the beads in the preceding row. 

McDowell provided a graph-theoretic interpretation of this characterization. Consider the \emph{additive residue graph} $\mathcal{G}_p$, where the vertices are the residues modulo $p$, and for each residue $x$ and every $1 \leq i \leq p - 1$, there is an edge from $x$ to $x + i$ modulo $p$, labelled $i$. Then every aligned abacus corresponds to a walk on the additive residue graph, where we start at $0$, and for every row of $i$ gaps and $p - i$ beads, we walk along the unique edge labelled $i$ from the current vertex. The labels of the edges on this walk must be nondecreasing; conversely, any walk on $\mathcal{G}_p$ starting at $0$ which has nondecreasing edge labels corresponds to a unique aligned abacus. 

For every row of the abacus, the parts contributed by this row modulo $p$ are the residue visited at the end of its corresponding edge. In particular, the partition corresponding to an aligned abacus is a $p'$-partition if and only if the walk never returns to $0$. In light of this property, we call a walk on $\mathcal{G}_p$ \emph{valid} if it begins at $0$, has nondecreasing edge labels, and never returns to $0$. 

McDowell proved the following interpretation of maximal $p$-core $p'$-partitions in terms of walks on $\mathcal{G}_p$. 

\begin{theorem}[{\cite[Theorem 4.8]{McD22}}] \label{Lambda characterization}
    There is a unique $p$-core $p'$-partition of maximal size, which corresponds to the unique longest valid walk on $\mathcal{G}_p$. 
\end{theorem}

Furthermore, McDowell also proved a characterization of the longest valid walk on $\mathcal{G}_p$. 

\begin{theorem}[{\cite[Theorem 4.6]{McD22}}] \label{longest walk characterization}
    There is a unique longest valid walk on $\mathcal{G}_p$, and for every $1 \leq i \leq p - 1$, this walk must have an edge labelled $i$ incident to the vertex $p - 1$. In other words, the longest walk must consist of the following segments:
    \begin{enumerate}
        \item Starting from $0$, take $p - 1$ edges labelled $1$ to reach $p - 1$. 
        \item For each $1 \leq i \leq p - 2$, starting from $p - 1$, take some number (possibly $0$) of edges labelled $i$, and then some number (possibly $0$) of edges labelled $i + 1$, to return to $p - 1$ (without reaching $0$). 
        \item Starting from $p - 1$, take $p - 2$ edges labelled $p - 1$ to reach $1$. 
    \end{enumerate}
\end{theorem}

\begin{example*}
    The maximal $3$-core $3'$-partition is $\Lambda_3 = (4, 2, 2, 1, 1)$, which corresponds to the following abacus. 
    \begin{center}
        \begin{tabular}{ccc}
            $0$ & $1$ & $2$ \\
            \hline
            $\cdot$ & $\circ$ & $\circ$ \\
            $\cdot$ & $\circ$ & $\circ$ \\
            $\cdot$ & $\cdot$ & $\circ$ \\
        \end{tabular}
    \end{center} 
    The corresponding walk on $\mathcal{G}_3$ is $(0, 1, 2, 1)$, consisting of two edges labelled $1$ and one edge labelled $2$. 

    
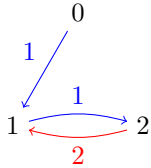
\begin{center}
        \begin{tikzpicture}
            \node (0) at (90:1) {$0$};
            \node (1) at (210:1) {$1$};
            \node (2) at (330:1) {$2$};
            \draw [->, blue] (0) -- (1) node [midway, above left] {$1$};
            \draw [->, blue, out=15, in=165] (1) to node [midway, above] {$1$} (2);
            \draw [->, red, out=195, in=345] (2) to node [midway, below] {$2$} (1);
        \end{tikzpicture}
    \end{center}
\end{example*}

We will use the characterization of the walk corresponding to $\Lambda_p$ given by Theorems \ref{Lambda characterization} and \ref{longest walk characterization} to bound $|\Lambda_p|$. 

\subsection{Some Lemmas}

We establish some notation to describe the walk in Theorem \ref{longest walk characterization}. For each $1 \leq i \leq p - 2$, define $x_i^{\max}$ and $y_i^{\max}$ to be the inverses of $i$ and $-(i + 1)$ mod $p$ respectively (such that $x_i^{\max}$ and $y_i^{\max}$ are in $\{1, 2, \ldots, p - 1\}$) --- in other words, $x_i^{\max}$ is the number of steps required to walk from $p - 1$ to $0$ using edges labelled $i$, and $y_i^{\max}$ is the number of steps required to walk from $0$ to $p - 1$ using edges labelled $i + 1$. Then the segment described in (2) corresponding to $i$ avoids $0$ if and only if it consists of strictly less than $x_i^{\max}$ edges labelled $i$ and strictly less than $y_i^{\max}$ edges labelled $i + 1$. 

Define $(x_i, y_i)$ to be the solution to \[ix + (i + 1)y \equiv 0 \pmod{p}\] over positive integers $x \leq x_i^{\max}$ and $y \leq y_i^{\max}$, with minimal $x + y$ (which exists because $(x_i^{\max}, y_i^{\max})$ is a solution, and is unique because $x + y$ uniquely determines $x$ and $y$ mod $p$). Then the segment of the walk described in (2) corresponding to $i$ must consist of exactly $x_i^{\max} - x_i$ edges labelled $i$, followed by $y_i^{\max} - y_i$ edges labelled $i + 1$ --- this is because the segment must consist of $x_i^{\max} - x$ edges labelled $i$ and $y_i^{\max} - y$ edges labelled $i + 1$ for some $x$ and $y$ satisfying these conditions, and its length is maximized when $x + y$ is minimized. The reason to consider the \emph{subtractions} $x_i$ and $y_i$ from the upper bounds on the path lengths (given by the fact that the path cannot visit $0$) rather than the lengths themselves is that proving $|\Lambda_p| \sim \frac{1}{24}p^6$ essentially amounts to showing that the subtractions are usually small; these subtractions turn out to be easier to work with than the original lengths.  

In order to use this characterization to bound $|\Lambda_p|$, we first establish some results about the size of $(x_i, y_i)$ over all possible $i$. 

\begin{lemma} \label{fraction representations}
    For any $1 \leq i \leq p - 2$, either \[i \equiv -\frac{s}{r + s} \pmod{p}\] for some relatively prime $0 < r, s < \sqrt{p}$, or \[i \equiv \frac{s}{r - s} \pmod{p}\] for some relatively prime $0 < r, s < \sqrt{p}$ which are not both $1$. Furthermore, there is at most one such pair $(r, s)$ with $i \equiv -\frac{s}{r + s} \pmod{p}$, and at most one such pair with $i \equiv \frac{s}{r - s} \pmod{p}$. 
\end{lemma}

\begin{proof}
    To prove the existence of such a pair $(r, s)$, consider the residues \[ia + (i + 1)b \pmod{p}\] across all pairs of integers $(a, b)$ with $0 < a, b < \sqrt{p} + 1$. Since there are strictly more than $p$ such pairs, some residue must be repeated; therefore we have \[i(a_1 - a_2) + (i + 1)(b_1 - b_2) \equiv 0 \pmod{p}\] for some $(a_1, b_1) \neq (a_2, b_2)$. We cannot have $a_1 = a_2$ or $b_1 = b_2$, as one equality would imply the other (since $i$ and $i + 1$ are nonzero mod $p$), so $0 < |a_1 - a_2|, |b_1 - b_2| < \sqrt{p}$. 

    If $a_1 - a_2$ and $b_1 - b_2$ have the same sign, then we have \[ir + (i + 1)s \equiv 0 \pmod{p}\] and therefore $i \equiv -\frac{s}{r + s} \pmod{p}$ for some $0 < r, s < \sqrt{p}$. Meanwhile, if $a_1 - a_2$ and $b_1 - b_2$ have opposite sign, then we have \[ir - (i + 1)s \equiv 0 \pmod{p}\] and therefore $i \equiv \frac{s}{r - s} \pmod{p}$ for some $0 < r, s < \sqrt{p}$. In either case, if $r$ and $s$ are not relatively prime, we can divide both by $\gcd(r, s)$ to get the desired result. 

    To prove uniqueness, note that if there were two pairs $(r_1, s_1)$ and $(r_2, s_2)$ corresponding to the same value of $i$ (in either case), then we would have \[\frac{r_1}{s_1} \equiv \frac{r_2}{s_2} \pmod{p},\] which would imply $r_1s_2 - s_1r_2 \equiv 0 \pmod{p}$. But we have \[|r_1s_2 - s_1r_2| < \max(r_1s_2, s_1r_2) < p,\] so we must have $r_1s_2 = s_1r_2$; then since $\gcd(r_1, s_1) = \gcd(r_2, s_2) = 1$, this implies $(r_1, s_1) = (r_2, s_2)$. 
\end{proof}

Let $S$ denote the set of residues $1 \leq i \leq p - 2$ for which $i \equiv \frac{s}{r - s} \pmod{p}$ for some relatively prime $0 < r, s < \sqrt{p}$, and let $T$ denote the set of residues $1 \leq i \leq p - 2$ which cannot be written in this form. Then by Lemma \ref{fraction representations}, every residue $i \in T$ can instead be written as $i \equiv -\frac{s}{r + s} \pmod{p}$ for some relatively prime $0 < r, s < \sqrt{p}$. 

\begin{lemma} \label{S case}
    If $i \in S$ with $i \equiv \frac{s}{r - s} \pmod{p}$, then \[\frac{p}{\max(r, s)} < x_i + y_i < \frac{p}{\max(r, s)} + \max(r, s) - 1.\] 
\end{lemma}

\begin{proof}
    We can rewrite the equation $ix + (i + 1)y \equiv 0 \pmod{p}$ as \[sx + ry \equiv 0 \pmod{p}.\] Assume without loss of generality that $s > r$ (the proof of the case $s < r$ is symmetric). Then we must have \[s(x + y) > sx + ry \geq p\] and therefore $x + y > \frac{p}{s}$ for any positive integer solution $(x, y)$. On the other hand, by the Chinese Remainder Theorem there exists some $0 < a < rs$ such that $s \mid (p - a)$ and $r \mid a$; then \[(x, y) = \left(\frac{p - a}{s}, \frac{a}{r}\right)\] is a solution, and satisfies $x < \frac{p}{s}$ and $y \leq s - 1$. 
    
    It remains to check that this solution satisfies $x \leq x_i^{\max}$ and $y \leq y_i^{\max}$. To check the first statement, we have $x_i^{\max} \equiv \frac{r - s}{s} \pmod{p}$, so \[sx_i^{\max} - r + s \equiv 0 \pmod{p}.\] Since the left-hand side is positive, it must then be at least $p$, which implies \[x_i^{\max} \geq \frac{p + r - s}{s} > \frac{p}{s} - 1.\] Then since $x$ and $x_i^{\max}$ are both integers and $x < x_i^{\max} - 1$, we must have $x \leq x_i^{\max}$. 

    Similarly, we have $y_i^{\max} \equiv -\frac{r - s}{r} \pmod{p}$, so \[ry_i^{\max} + r - s \equiv 0 \pmod{p}.\] If $r = 1$, then $y_i^{\max} = s - 1 \geq y$. Otherwise, we cannot have $ry_i^{\max} + r - s = 0$ (as this implies $r \mid s$, but $r$ and $s$ are relatively prime), so $ry_i^{\max} + r - s \geq p$, and \[y_i^{\max} \geq \frac{p + s - r}{r} > \frac{p}{r} > s.\] 
    
    This means \[x_i + y_i \leq x + y < \frac{p}{s} + s - 1,\] as desired. 
\end{proof}

\begin{remark*}
    In fact, it is always true that $(x_i, y_i) = (x, y)$ for the solution described above --- we must have $sx_i + ry_i = p$, since otherwise $x_i + y_i > \frac{2p}{s} > \frac{p}{s} + s$. Then $(x_i, y_i) = (\frac{p - a}{s}, \frac{a}{r})$ for some positive integer $a$, and $x_i + y_i$ is minimized when $a$ is minimized, which occurs at our choice of $a$. Note that the stronger bound \[x_i + y_i < \frac{p}{\max(r, s)} + \max(r, s) - \min(r, s)\] is true as well, using the fact that $x_i + y_i = \frac{p}{s} + a\left(\frac{1}{r} - \frac{1}{s}\right) < \frac{p}{s} + s - r$. However, this does not have a large impact on the resulting bounds, so we use the weaker bound to simplify computations. 
\end{remark*}

\begin{lemma} \label{T case}
    If $i \in T$ with $i \equiv -\frac{s}{r + s} \pmod{p}$, then $x_i + y_i \leq r + s$. 
\end{lemma}

\begin{proof}
    First, the equation $ix + (i + 1)y \equiv 0 \pmod{p}$ can be rewritten as \[-sx + ry \equiv 0 \pmod{p},\] so $(x, y) = (r, s)$ is a solution. To check that it satisfies $x \leq x_i^{\max}$, we have $x_i^{\max} \equiv -\frac{r + s}{s} \pmod{p}$, so \[sx_i^{\max} + r + s \equiv 0 \pmod{p}.\] But if $x_i^{\max} \leq r - 1$, then we would have \[sx_i^{\max} + r + s \leq rs + r < \sqrt{p}(\sqrt{p} - 1) + \sqrt{p} = p,\] since we cannot have $r = s = \lfloor \sqrt{p}\rfloor$ (as $r$ and $s$ are relatively prime); this is a contradiction. The same proof shows that we cannot have $y_i^{\max} \leq s - 1$. 

    This implies $(r, s)$ is a valid solution for $(x, y)$, which means \[x_i + y_i \leq r + s,\] as desired.  
\end{proof}

\begin{remark*}
    It is again true that we always have $(x_i, y_i) = (r, s)$. To prove this, assume for contradiction there is a solution $(x, y)$ with $x + y < r + s$. Then $x$ and $y$ cannot both be less than $\sqrt{p}$, as otherwise by uniqueness in Lemma \ref{fraction representations}, we would have $\frac{1}{\gcd(x, y)}(x, y) = (r, s)$. On the other hand, we must have $x < r + \sqrt{p}$ and $y < s + \sqrt{p}$ (as otherwise $x$ or $y$ would be greater than $r + s$). Then considering the equation \[i(x - r) + (i + 1)(y - s) \equiv 0 \pmod{p}\] gives that $i \in S$ (since $x - r$ and $y - s$ have opposite sign, and $0 < |x - r|, |y - s| < \sqrt{p}$), contradiction. 
\end{remark*}

Finally, we will also need the following number theoretic lemma (which arises from the fact that the residues $i \in S$ correspond exactly to pairs of distinct relatively prime $(r, s)$ with $0 < r, s < \sqrt{p}$). 

\begin{lemma} \label{totient bound}
    For any positive integer $n$, we have \[\sum_{m = 1}^n \frac{\varphi(m)}{m} > \frac{3}{5}n - 6,\] where $\varphi$ denotes the Euler totient function.  
\end{lemma}

\begin{remark*}
    In general, one can obtain better lower bounds than Lemma \ref{totient bound}. However, for our purposes, we are content with this simple argument. 
\end{remark*}

\begin{proof}
    By using the fact that $\varphi = \mu * \operatorname{id}$ (where $*$ denotes Dirichlet convolution, $\mu$ denotes the M\"obius function, and $\operatorname{id}$ denotes the identity map sending every positive integer to itself) and swapping the order of summation, we have
    \begin{align*}
        \sum_{m = 1}^n \frac{\varphi(m)}{m} &= \sum_{m = 1}^n \frac{1}{m}\sum_{d \mid m} \mu(d)\cdot \frac{m}{d} \\
        &= \sum_{d = 1}^\infty \frac{\mu(d)}{d}\cdot \floor{\frac{n}{d}} \\
        &> n\sum_{d = 1}^\infty \frac{\mu(d)}{d^2} - \sum_{d = 1}^n \frac{1}{d} - n\sum_{d = n + 1}^\infty \frac{1}{d^2}. 
    \end{align*}
    To calculate the first term, let $1$ denote the function sending every positive integer to $1$. Then $(\mu * 1)(n)$ is $1$ if $n = 1$ and $0$ otherwise, so $\sum_{d = 1}^\infty \frac{\mu(d)}{d^2} \cdot \sum_{d = 1}^\infty \frac{1}{d^2} = 1$, and therefore \[\sum_{d = 1}^\infty \frac{\mu(d)}{d^2} = \frac{1}{\zeta(2)} = \frac{6}{\pi^2}.\] This provides the bound \[\sum_{m = 1}^n \frac{\varphi(m)}{m} > n\cdot \frac{6}{\pi^2} - \log n - 2 > \frac{3}{5}n - 6.\qedhere\] 
\end{proof}

\subsection{Proof of Theorem \ref{main theorem}}

Suppose that $p > 10^6$ is prime. We use the above lemmas to establish bounds on $|\Lambda_p|$. 

For convenience, we first define a few more pieces of data associated with the abacus corresponding to $|\Lambda_p|$. Recall that $b_1$, \ldots, $b_{p - 1}$ denote the bead multiplicities of $\Lambda_p$. By Lemma \ref{right alignment}, we can also define the \emph{row multiplicities} $m_1$, \ldots, $m_{p - 1}$, where $m_i$ is the number of rows consisting of $i$ gaps followed by $p - i$ beads. Then we have \[b_i = m_1 + m_2 + \cdots + m_i\] for all $1 \leq i \leq p - 1$. By Theorems \ref{Lambda characterization} and \ref{longest walk characterization}, we have $m_1 = p - 1$, $m_{p - 1} = p - 2$, and for all $2 \leq i \leq p - 2$, \[m_i = y_{i - 1}^{\max} - y_{i - 1} + x_i^{\max} - x_i = p - y_{i - 1} - x_i.\] 

As noted by McDowell, the row multiplicities have useful symmetry --- we have $(x_i, y_i) = (y_{p - 1 - i}, x_{p - 1 - i})$ for all $1 \leq i \leq p - 2$ (since the conditions used to define them are identical), which means $m_i = m_{p - i}$ for all $2 \leq i \leq p - 2$ \cite[Corollary 4.9]{McD22}. 

Now define $d_i := p - m_i$ for all $1 \leq i \leq p - 1$, so that $d_1 = 1$, $d_{p - 1} = 2$, and $d_i = y_{i - 1} + x_i$ for all $2 \leq i \leq p - 2$. Finally, define $c_i := d_1 + d_2 + \cdots + d_i$, so that \[b_i = m_1 + \cdots + m_i = ip - c_i,\] and define \[c := \sum_{i = 1}^{p - 2} (x_i + y_i) = 2 + \sum_{i = 2}^{p - 2} (y_{i - 1} + x_i).\] Then we have $c_{p - 1} = c + 1$, while the symmetry property implies that \[c_i + c_{p - 1 - i} = c\] for all $1 \leq i \leq p - 2$. 

Intuitively, the variables $d_i$ and $c_i$ represent how far $m_i$ and $b_i$ are from their crude upper bounds --- in particular, McDowell obtained the upper bound (\ref{McD upper bound}) by using the fact that $m_i \leq p - 2$ for all $i \geq 2$. Our strategy is therefore to show that these subtractions are small. More precisely, we use the previous lemmas to establish upper and lower bounds on $c$ which are both on the order of $p\sqrt{p}$, and then translate these bounds on $c$ to bounds on $|\Lambda_p|$ via Lemma \ref{size from bi}.

\begin{lemma} \label{upper bound on c}
    We have\footnote{This lemma is true for all $p \geq 17$.} that $c < \frac{11}{3}p\sqrt{p}$. 
\end{lemma}

\begin{proof}
    Separate the sum as \[c = \sum_{i = 1}^{p - 2} (x_i + y_i) = \sum_{i \in S} (x_i + y_i) + \sum_{i \in T} (x_i + y_i).\] For the sum over $i \in S$, if $i \equiv \frac{s}{r - s} \pmod{p}$ as given by Lemma \ref{fraction representations} with $\max(r, s) = m$, then we have \[x_i + y_i \leq \frac{p}{m} + m - 1\] by Lemma \ref{S case}. For each $1 \leq m < \sqrt{p}$, there are less than $2m$ pairs $(r, s)$ with $\max(r, s) = m$, so we have 
    \begin{align*}
        \sum_{i \in S} (x_i + y_i) &< \sum_{m = 1}^{\floor{\sqrt{p}}} 2m\left(\frac{p}{m} + m - 1\right) \\
        &< 2p\sqrt{p} + 4\sum_{m = 1}^{\floor{\sqrt{p}}} \binom{m}{2} \\
        &= 2p\sqrt{p} + 4\binom{\floor{\sqrt{p}} + 1}{3} \\
        &< \frac{8}{3}p\sqrt{p}.
    \end{align*}
    Meanwhile, for the sum over $i \in T$, if $i \equiv -\frac{s}{r + s} \pmod{p}$ as given by Lemma \ref{fraction representations}, then we have $x_i + y_i \leq r + s$ by Lemma \ref{T case}. Since $\gcd(r, s) = 1$ for any such pair $(r, s)$, for any fixed $s > 2$ there are less than $\sqrt{p} - 1$ values of $r$ (as $r$ cannot equal $s$); meanwhile for $s = 2$ there are less than $\sqrt{p} - 2$ values of $r$ (as $r$ cannot equal $2$ or $4$, which are both less than $\sqrt{p}$), and for $s = 1$ there are less than $\sqrt{p}$ values of $r$. Combining these estimates, we have \[\sum_{i \in T} (x_i + y_i) \leq \sum_{(r, s)} (r + s) < 2(\sqrt{p} - 1)\cdot \frac{\sqrt{p}(\sqrt{p} + 1)}{2} < p\sqrt{p}.\] Combining the two sums gives $c < \frac{11}{3}p\sqrt{p}$.
\end{proof}

We now obtain a lower bound on $c$. 

\begin{lemma} \label{lower bound on c}
    We have that $c > \frac{6}{5}p\sqrt{p} - 16p$.
\end{lemma}

\begin{proof}
    We have \[c \geq \sum_{i \in S} (x_i + y_i) > \sum_{(r, s)} \frac{p}{\max(r, s)}\] by Lemma \ref{S case}. For each $2 \leq m < \sqrt{p}$, there are $2\varphi(m)$ pairs of relatively prime distinct $(r, s)$ with $0 < r, s < \sqrt{p}$ and $\max(r, s) = m$, and by uniqueness in Lemma \ref{fraction representations}, each corresponds to a different value of $i \in S$. This means \[c > \sum_{m = 2}^{\floor{\sqrt{p}}} \frac{p}{m}\cdot 2\varphi(m) > 2p\left(\frac{3}{5}\floor{\sqrt{p}} - 7\right) > \frac{6}{5}p\sqrt{p} - 16p,\] using the estimate given by Lemma \ref{totient bound}. 
\end{proof}

Finally, we will also need the following upper bound on an intermediate sum. 

\begin{lemma} \label{silly 18 lemma}
    We have\footnote{This lemma is true for all $p > 256$.} that $c_{\floor{p/18}} < \frac{2}{5}p\sqrt{p} + p$.
\end{lemma}

\begin{proof}
    We have \[c_{\floor{p/18}} = 1 + \sum_{i = 2}^{\floor{p/18}} (y_{i - 1} + x_i) < \sum_{i = 1}^{\floor{p/18}} (x_i + y_i).\] By Lemmas \ref{S case} and \ref{T case}, if $i \in T$ then $x_i + y_i < 2\sqrt{p}$; on the other hand, if $i \in S$ is written as $i \equiv \frac{r}{r - s} \pmod{p}$ with $\max(r, s) = m$, then $x_i + y_i < \frac{p}{m} + m - 1$. For each $1 \leq m < \sqrt{p}$, there are at most $m$ pairs $(r, s)$ with $\max(r, s) = m$ corresponding to indices $i < \frac{p}{18}$, since swapping $r$ and $s$ maps $i \mapsto p - 1 - i$. 

    Let $t := \floor{\sqrt{p}/3}$. Then we can bound one entry $x_i + y_i$ above by $p$, two entries by $\frac{p}{2} + 1$, three by $\frac{p}{3} + 2$, and so on, up to $t$ entries by $\frac{p}{t} + (t - 1)$ --- note that $1 + 2 + \cdots + t > \frac{t^2}{2} > \frac{p}{18}$, and $\frac{p}{t} + (t - 1) > 2\sqrt{p}$. This means 
    \begin{align*}
        \sum_{i = 1}^{\floor{p/18}} (x_i + y_i) &< \sum_{m = 1}^t m\left(\frac{p}{m} + m - 1\right) \\
        &= tp + 2\binom{t + 1}{3} \\
        &< p\left(\frac{\sqrt{p}}{3} + 1\right) + \frac{1}{3}\left(\frac{\sqrt{p}}{3} + 1\right)^3. 
    \end{align*}
    This expression is less than $\frac{2}{5}p\sqrt{p} + p$, as $p > 10^6 > 256$. 
\end{proof}

Now we are in a position to establish the claimed inequalities in Theorem \ref{main theorem}.

\begin{proof}[Proof of Theorem \ref{main theorem}]
    In order to establish both bounds, we use the formula 
    \begin{align}
        |\Lambda_p| = -\frac{1}{2}\left(\sum_{i = 1}^{p - 1} b_i\right)^2 + \frac{p}{2}\sum_{i = 1}^{p - 1} b_i^2 + \sum_{i = 1}^{p - 1} \left(i - \frac{p - 1}{2}\right)b_i \label{size formula}
    \end{align}
    given by Lemma \ref{size from bi}. For both directions, we first have 
    \begin{align}
        \sum_{i = 1}^{p - 1} b_i = \sum_{i = 1}^{p - 1} (ip - c_i) = \frac{1}{2}p^2(p - 1) - \frac{1}{2}pc - 1, \label{sum bi}
    \end{align}
    using the fact that $c_i + c_{p - 1 - i} = c$ for all $1 \leq i \leq p - 2$, and $c_{p - 1} = c + 1$. 

    We first prove the lower bound \[|\Lambda_p| > \frac{1}{24}p^6 - p^5\sqrt{p}.\] For the first term of (\ref{size formula}), we use the bound \[-\frac{1}{2}\left(\sum_{i = 1}^{p - 1} b_i\right)^2 > -\frac{1}{2}\left(\frac{p^3 - pc}{2}\right)^2 = -\frac{1}{8}p^6 + \frac{1}{4}p^4c - \frac{1}{8}p^2c^2\] from (\ref{sum bi}). For the second term of (\ref{size formula}), we use the bound \[b_i^2 = (ip - c_i)^2 > i^2p^2 - 2ipc\] for all $1 \leq i \leq p - 1$ (this is clear for $1 \leq i \leq p - 2$ as $c_i \leq c$, while for $i = p - 1$, it follows from the fact that $(c + 1)^2 > 2p(p - 1)$,  as $p > 10^6 > 250$, using the bound in Lemma \ref{lower bound on c}), which gives \[\frac{p}{2}\sum_{i = 1}^{p - 1} b_i^2 > \frac{1}{12}p^4(p - 1)(2p - 1) - \frac{1}{2}p^3(p - 1)c.\] Finally, the third term in (\ref{size formula}) must be positive, as for each $1 \leq i \leq \frac{p - 1}{2}$ we have $b_{p - 1 - i} \geq b_i$, which means \[\left(i - \frac{p - 1}{2}\right)b_i + \left(p - 1 - i - \frac{p - 1}{2}\right)b_{p - 1 - i} \geq 0.\] Combining these bounds and using (\ref{size formula}), we have \[|\Lambda_p| > \frac{1}{24}p^6 - \frac{1}{4}p^4c - \frac{1}{8}p^2c^2 - \frac{1}{4}p^5.\] Using the bound $c < \frac{11}{3}p\sqrt{p} < 4p\sqrt{p}$, we have \[|\Lambda_p| > \frac{1}{24}p^6 - \frac{11}{12}p^5\sqrt{p} - \frac{9}{4}p^5 > \frac{1}{24}p^6 - p^5\sqrt{p},\] where the second inequality holds as $p > 10^6 > 729$. 

    We now prove the upper bound \[|\Lambda_p| < \frac{1}{24}p^6 - \frac{1}{200}p^5\sqrt{p}.\] The first term of (\ref{size formula}) is given by 
    \begin{align*}
        -\frac{1}{2}\left(\sum_{i = 1}^{p - 1} b_i\right)^2 &< -\frac{1}{8}(p^3 - pc - p^2 - 2)^2 \\
        &< -\frac{1}{8}p^6 + \frac{1}{4}p^4c - \frac{1}{8}p^2c^2 + \frac{1}{4}p^5
    \end{align*}
    by (\ref{sum bi}). To bound the second term of (\ref{size formula}), we pair terms --- for each $1 \leq i \leq \frac{p - 1}{2}$, we have
    \begin{align*}
        b_i^2 + b_{p - 1 - i}^2 &= (ip - c_i)^2 + ((p - 1 - i)p - c_{p - 1 - i})^2 \\
        &< i^2p^2 + (p - 1 - i)^2p^2 - p(c - 2c_i)(p - 1 - 2i) - cp(p - 1) + c^2,
    \end{align*}
    where we used the fact that $c_i^2 + c_{p - 1 - i}^2 < (c_i + c_{p - 1 - i})^2 = c^2$. For all $1 \leq i \leq \frac{p - 1}{2}$, we have $c - 2c_i \geq 0$ and $p - 1 - 2i \geq 0$. Furthermore, for all $i < \frac{p}{18}$, by Lemmas \ref{lower bound on c} and \ref{silly 18 lemma} we have \[c - 2c_i > \frac{2}{5}p\sqrt{p} - 18p,\] and $p - 1 - 2i \geq \frac{8p}{9} - 1$. Finally, we have $b_{p - 1}^2 < p^2(p - 1)^2$. Combining these bounds gives 
    \begin{align*}
        \sum_{i = 1}^{p - 1} b_i^2 &< \frac{p^3(p - 1)(2p - 1)}{6} - \frac{p - 2}{2}\cdot cp(p - 1) - p\left(\frac{1}{18}p - 1\right)\left(\frac{2}{5}p\sqrt{p} - 18p\right)\left(\frac{8}{9}p - 1\right) + \frac{p - 2}{2}\cdot c^2 \\
        &< \frac{1}{3}p^5 - \frac{1}{2}p^3c - \frac{1}{60}p^4\sqrt{p} + \frac{1}{2}pc^2 + \frac{1}{4}p^4 + p^2c + \frac{1}{6}p^3,
    \end{align*}
    using the bounds $\frac{1}{18}p - 1 > \frac{1}{20}p$ and $\frac{8}{9}p - 1 > \frac{5}{6}p$, which hold as $p > 10^6 > 180$. So the second term in (\ref{size formula}) is bounded above by \[\frac{p}{2}\sum_{i = 1}^{p - 1} b_i^2 < \frac{1}{6}p^6 - \frac{1}{4}p^4c - \frac{1}{120}p^5\sqrt{p} + \frac{1}{4}p^2c^2 + \frac{1}{8}p^5 + \frac{1}{2}p^3c + \frac{1}{12}p^3.\] Finally, the third term in (\ref{size formula}) is bounded above by \[\sum_{i = 1}^{p - 1} \left(i - \frac{p - 1}{2}\right)b_i < \sum_{i = 1}^{p - 1} i\cdot ip < \frac{1}{3}p^4.\] Combining these bounds and using (\ref{size formula}), we have \[|\Lambda_p| < \frac{1}{24}p^6  - \frac{1}{120}p^5\sqrt{p} + \frac{1}{8}p^2c^2 + \frac{5}{8}p^5 + \frac{1}{2}p^3c + \frac{1}{3}p^4+ \frac{1}{12}p^3.\] Using the bound $c < 4p\sqrt{p}$ then gives \[|\Lambda_p| < \frac{1}{24}p^6 - \frac{1}{120}p^5\sqrt{p} + \frac{23}{8}p^5 + 2p^4\sqrt{p} + \frac{1}{3}p^4 + \frac{1}{12}p^3 < \frac{1}{24}p^6 - \frac{1}{200}p^5\sqrt{p},\] where the second inequality holds as $p > 10^6$. 
\end{proof}


\end{document}